\documentclass[11pt]{article}
\usepackage[utf8]{inputenc}
\usepackage{soul}

\usepackage{subcaption}

\newcommand{\required}[1]{\section*{\hfil \sharp1\hfil}}

\usepackage{verbatim,url}
\usepackage{graphicx, epsfig}
\graphicspath{{Figs/}}
\usepackage{bbm}
\usepackage{amssymb,amsmath,amsfonts,amsthm}
\usepackage{upref}

\newcommand{\Beq}{\begin{equation}}
\newcommand{\Eeq}{\end{equation}}
\newcommand{\be}{\begin{equation}}
\newcommand{\ee}{\end{equation}}
\newcommand{\beq}{\begin{equation*}}
\newcommand{\eeq}{\end{equation*}}
\newcommand{\bal}{\begin{align}}
\newcommand{\eal}{\end{align}}

\newcommand{\s}{\mathcal S}

\newcommand{\BP}{\mathbb P}

\newcommand{\CH}{\mathcal H}
\newcommand{\br}{\mathbb{R}}
\newcommand{\N}{\mathbb{N}}
\newcommand{\ik}{\varphi}
\newcommand{\pa}{\partial}

\newcommand{\al}{\alpha}

\newcommand{\dd}{\text{d}}
\newcommand{\fks}{\phi_{l}}
  
\newcommand{\alst}{\alpha_{\star}} 
\newcommand{\ga}{\gamma}  
\newcommand{\chx}{\check x}  
\newcommand{\chy}{\check y}  
\newcommand{\Ner}{N_\e^{\text{rec}}}

\newcommand{\Hc}{\mathcal{H}}

\newcommand{\Eb}{\mathbb{E}}

\newcommand{\Pb}{\mathbb{P}}

\newcommand{\Rb}{\mathbb{R}}
\newcommand{\Sb}{\mathbb{S}}

\newcommand{\Zb}{\mathbb{Z}}

\newcommand{\Ab}{\vec{\alpha}}

\newcommand{\e}{\epsilon}
\newcommand{\de}{\delta}
\newcommand{\cx}{Q}
\newcommand{\DTB}{\text{DTB}}

\newcommand{\frec}{f_\e^{\text{rec}}}
\numberwithin{equation}{section}

\newtheorem{theorem}{Theorem}[section]
\newtheorem{lemma}[theorem]{Lemma}

\newtheorem{definition}[theorem]{Definition}
\newtheorem{corollary}[theorem]{Corollary}
\newtheorem{assumption}[theorem]{Assumption}

\theoremstyle{definition}

\usepackage{xcolor}

\usepackage{authblk} 

\title{Local reconstruction analysis of inverting the Radon transform in the plane from noisy discrete data}
\author[$\dagger$]{Anuj Abhishek} 
\author[$*$]{Alexander Katsevich}
\author[$+$]{James W. Webber}

\affil[$\dagger$]{\small Department of Mathematics, Applied Mathematics and Statistics, Case Western Reserve University, USA\par axa1828@case.edu}

\affil[$*$]{\small Department of Mathematics, University of Central Florida, USA\par alexander.katsevich@ucf.edu}
\affil[$+$]{\small Brigham and Women's Hospital and Harvard University, USA\par jwebber5@bwh.harvard.edu }
\date{}
\begin{document}
\maketitle

\begin{abstract} 
In this paper, we investigate the reconstruction error, $N_\e^{\text{rec}}(x)$, when a linear, filtered back-projection (FBP) algorithm is applied to noisy, discrete Radon transform data with sampling step size $\epsilon$ in two-dimensions. Specifically, we analyze $N_\e^{\text{rec}}(x)$ for $x$ in small, $O(\e)$-sized neighborhoods around a generic fixed point, $x_0$, in the plane, where the measurement noise values, $\eta_{k,j}$ (i.e., the errors in the sinogram space), are random variables. The latter are independent, but not necessarily identically distributed.   We show, under suitable assumptions on the first three moments of the $\eta_{k,j}$, that the following limit exists:
$N^{\text{rec}}(\chx;x_0) = \lim_{\e\to0}N_\e^{\text{rec}}(x_0+\e\chx)$, for $\check x$ in a bounded domain. Here, $N_\e^{\text{rec}}$ and $ N^{\text{rec}}$ are viewed as continuous random variables, and the limit is understood in the sense of distributions. Once the limit is established, we prove that $N^{\text{rec}}$ is a zero mean Gaussian random field and compute explicitly its covariance. In addition, we validate our theory using numerical simulations and pseudo random noise.
\end{abstract}

\section{Introduction}
An important practical goal in computed tomography (CT) applications is to understand the relation between resolution in image reconstruction and the rate of sampling at which the (noisy) discrete tomographic data is collected. In particular, often one needs to understand at what resolution the \textit{singularities} (non-smoothness) of the original object appear in the reconstructed images. The question of resolution of singularities is extremely important in many applications, such as medical imaging, nondestructive evaluation, metrology, luggage and cargo scanning for threat detection, to name a few. 

As an example, consider medical imaging. By itself, this is a very diverse area. Different areas of the body may have a variety of pathologies, each leading to specific requirements for image quality and, in particular, image resolution. We will briefly outline one such task, namely detection and assessment of lung tumors in CT images. Typically, malignant tumors (e.g., lung nodules) have rougher boundaries than benign tumors \cite{dhara2016combination,dhara2016differential}. When diagnosing whether a lung nodule is malignant or benign, the roughness of the nodule boundary is a critical factor. Typically, reconstructions from discrete X-ray CT data will lead to images in which the singularities are smoothed to some extent. In particular, a rough boundary in the original object (tumor) may appear as a smoothed boundary in the reconstructed image. This can lead to a cancerous tumor being misdiagnosed as a benign nodule. On the other hand, due to the random noise in the data, the boundary of a benign nodule may appear rougher than it actually is, which may again result in misdiagnosis.  This example illustrates the need to accurately quantify the effects of both data discretization and random noise on {\it local} tomographic reconstruction.
%Moreover, the situation is further complicated by the fact that the CT data is invariably corrupted with random noise. Not accounting for the presence of noise in the data may result in interpreting the roughness in the reconstructed image due to the effect of random noise as essentially a feature of the original object. This might result in a practitioner misclassifying a benign tumor as a cancerous one, which would again be harmful for a patient. Thus, it behooves us to develop a theory that will quantify the accuracy of tomographic reconstruction of objects with rough boundaries (such as malignant tumors) under the twin effects of data-discretization and noise in the data and to apply the understanding so developed to  applications such as improving cancer diagnostics and assessment.

In a series of articles \cite{Katsevich2017a,kat19a,Katsevich2020b,Katsevich2020a,Katsevich2021a}, the authors developed a novel approach, called \textit{local reconstruction analysis} (LRA), to analyze the resolution with which the singularities (i.e., jump discontinuities) of an object, $f$, are reconstructed from discrete tomographic data in a deterministic setting, i.e., in the absence of noise. In those articles the singularities of $f$ are assumed to lie on a smooth curve, denoted $\s$. Later, in \cite{Katsevich2022a,Katsevich2023a,Katsevich_2024_BV}, LRA was extended to functions on the plane with jumps across rough boundaries (e.g., fractals). In \cite{Katsevich_aliasing_2023}, LRA theory was further advanced to include analysis of aliasing (ripple artifacts) at points away from the boundary.
%\jc{\red{As this is now aliasing, it's no longer ``local" resolution analysis? SASHA: ``Local" in the name refers to the fact that we are studying some property in an $\e$-nbhd of a point. This is exactly what was done in the aliasing analysis paper.}}

To illustrate the key idea of LRA, let us consider the reconstruction from discrete, noise-free CT data in $\Rb^2$ that is sampled in the the angular and affine variables at $O(\epsilon)$ step-sizes. LRA fully describes the behaviour of the reconstructed image in an $O(\epsilon)$-sized neighborhood of a generic point, $x_0\in\s$, where $\s$ denotes the singular support of $f$.  More precisely, let $f$ be represented by a real-valued function in $\Rb^2$, and let $\s\subset\br^2$ be a smooth curve. Let $\frec$ be an FBP reconstruction of $f$ from discretely sampled Radon Transform (RT) data with sampling step size, $\e$. Under appropriate conditions on $\s$, it is shown in \cite{Katsevich2022a,Katsevich2023a} that, for any $c>0$ one has
\be\label{DTB new use}
\lim_{\e\to0}\frec(x_0+\e\check x)=\Delta f(x_0)\DTB(\check x;x_0),\ |\check x|\le c,x_0\in\s,
\ee
where DTB, which stands for the \textit{Discrete Transition Behavior}, is an easily computable function independent of $f$, and depends only on the curvature of $\s$ at $x_0$. Here, $\Delta f(x_0)$ is the value of the jump of $f$ across $x_0$. 

When $\e$ is sufficiently small, the right-hand side of \eqref{DTB new use} is an accurate approximation of $\frec$ and the DTB function describes accurately the smoothing of the singularities of $f$ in $\frec$.

Among alternative approaches to study resolution, the most common one is based on sampling theory, see e.g. \cite{far04,nat93,pal95}. The usual assumption in these works is that the function $f$ to be reconstructed is essentially bandlimited in the Fourier domain, i.e., $f$ is smooth. A more recent approach to study reconstruction from discrete RT data uses tools of semiclassical analysis  \cite{Monard2021,stef20}. The assumptions here are more flexible than in classical sampling theory, but still assume that the data represent measurements of a semiclassically bandlimited signal. This means that either $f$ itself or the detector aperture function is semiclassically bandlimited and hence $C^{\infty}$ with non-compact support. In practical settings, these assumptions do not hold.

Two major sources of error in CT reconstruction are data discretization and the presence of random noise. LRA provides an accurate, local description of a reconstructed image from discretized data (e.g., as described by \eqref{DTB new use}). 
%From a statistical viewpoint, when the data-model includes random noise, the approach taken above has to be revisited.
In this article, we extend LRA to include the effects of random noise on the reconstructed image. 

We consider an additive noise model, $\hat{f} = Rf + \eta$, where $R$ denotes the classical 2D Radon transform, $\hat{f}$ is the measured sinogram, and $\eta$ is the noise. The entries, $\eta_{k,j}$, of $\eta$ are assumed to be independent, but not necessarily identically distributed. Due to linearity of $R^{-1}$, the recontruction error is $R^{-1}\eta$. We analyze this error in a discrete setting; specifically the effects of applying FBP purely to noise. To this end, let $N_\e^{\text{rec}}(x)$ be an FBP reconstruction only from $\eta$. Similarly to what was done earlier, we consider $O(\e)$-sized neighborhoods around any $x_0\in \mathbb{R}^2$. In this work, $x_0$ is not constrained to a 1D curve as in previous literature on LRA.  Let $C(D)$ be the space of continuous functions on $D$, where $D\subset\br^2$ is a bounded domain. In our first main result, we show, under suitable assumptions on the first three moments of the $\eta_{k,j}$, the boundary of $D$, and $x_0$, that the following limit exists:
%the goal here is to provide a complete description of the effects of noise in the data 
%focus on the part of the tomographic reconstruction that arises due to the presence of noise in the data and describe it locally in 
%in $O(\e)$-sized neighborhoods around any fixed point $x_0$ (i.e., not necessarily $x_0\in\s$). 
\be\label{noise rec 0}
N^{\text{rec}}(\chx;x_0)=\lim_{\e\to0}N_\e^{\text{rec}}(x_0+\e\chx),\ \check x\in D.
\ee
Here, $N^{\text{rec}}$ and $N_\e^{\text{rec}}$ are viewed as $C(D)$-valued random variables, and the limit is understood in the sense of distributions. Once the limit is established, we then go on to prove that $N^{\text{rec}}$ is a zero mean Gaussian random field (GRF) and compute explicitly its covariance. Numerical experiments, where the $\eta_{k,j}$ are simulated using a random number generator, are also also conducted to validate our theory. The results show an excellent match between our predicted theory and simulated reconstructions.

Taken together, \eqref{DTB new use} and \eqref{noise rec 0} provide a complete and accurate local description of the reconstruction error from discrete data in the presence of noise. This contrasts with global descriptions, which estimate the reconstruction error in some global (e.g., $L^2$) norm. The main novelty, and advantage of our two formulas is that they describe the reconstruction error at the scale of the data step-size ($\sim\hspace{-1mm}\e$), which is important in many CT applications (e.g., precise imaging of tumor boundaries in medical CT). No additional processing, such as smoothing at scales $\gg\e$ that may be necessary to establish convergence results, are applied. Thus, LRA allows statistical inference in $O(\e)$ size neighborhoods (i.e., at native resolution) of any $x_0$ (including boundary points) while accounting for both data-discretization and non-identically distributed random noise. To the best of the authors’ knowledge, this is the first ever result of such kind.

We will now survey some pertinent results in the existing literature related to reconstruction from noisy data and discuss how they compare to the proposed theory.
%and discuss how we begin filling the gap in the existing literature. 
%so as to make our theory particularly suitable for applications such as in medical diagnostics. 
%Questions of resolution and sampling in image reconstruction from tomographic data are important and have been studied quite well. The most common approach is based on sampling theory, see e.g. \cite{far04,nat93,pal95}. Usually, the assumption made in these works is that the function $f$ to be reconstructed is essentially bandlimited in the Fourier domain, i.e. $f$ is smooth. Additionally, the analysis of reconstruction is carried out without accounting for random noise that is invariably present in the data coming from practical applications. On the other hand, 
A statistical kernel-type estimator for the RT was derived in \cite{Tsybakov_1991,Tsybakov_92}, and the minimax optimal rate of convergence of the estimator to the ground truth, i.e. $f$, is established at a fixed point and in a global $L^2$ norm. The data are assumed to be collected at a random set of points rather than on a regular grid, which is the most common case in practice. In \cite{Bissantz2014}, the rate of convergence of the maximal deviation of an estimator from its mean is obtained for similar kernel-type estimators. Most notably, \cite{Bissantz2014,Tsybakov_1991,Tsybakov_92} establish an asymptotic convergence of an estimator to $f$ using additional smoothing at a scale $\de\gg\e$, which results in a significant loss of resolution in practice.
%in some global norm, but do not address local neighborhoods, which may be useful, e.g. for points where $f$ has jumps. 
In \cite{Cavalier_98,Cavalier_00}, the accuracy of pointwise asymptotically optimal (in the sense of minimax risk) estimation of a function from noise-free RT data sampled on a random grid is derived. 
In all the works cited above the functions being estimated are assumed to be sufficiently smooth. Also, they do not investigate the probability distribution of the reconstructed noise (error in the reconstruction) neither pointwise nor in a domain. 
%In applications, such as medical CT, these assumptions are not realistic. 

Approaches to study reconstruction in the framework of Bayesian inversion have been proposed as well \cite{Monard_19,Siltanen2003,Lassas_09}. In \cite{Monard_19}, the authors investigate global inversion in a continuous data setting assuming that the noise in the data is a Gaussian white noise. Using a Gaussian prior (i.e., with Tikhonov regularization), they establish the asymptotic normality of the posterior distribution and of the MAP estimator for quantities of the kind $\int f(x)\psi(x)\dd x$, where $\psi\in C^\infty$. See also \cite{Siltanen2003,Lassas_09} for a discussion of various aspects of Bayesian inversion.

An approach to study reconstruction errors using semiclassical analysis is developed in \cite{stef_23}. The goal in \cite{stef_23} is to analyze empirical spatial mean and variance of the noise in the inversion for a single experiment, as the sampling rate goes to zero. We analyze the reconstruction error value density and compute the expected value and covariance across multiple reconstructions.

Analysis of noise in reconstructed images is also an active area in more applied research, see e.g. \cite{Noo_noise_2008, Divel2020} and references therein. In these works, the methodology is mostly a combination of numerical and semi-empirical approaches, and theoretical analyses of the noise behaviour in small neighbohoods, such as those proposed here, are not provided.

%The work that comes closest in spirit to the analysis presented in this paper is , in which the authors have analysed the effect of noise on the inversion of Fourier Integral Operators (FIOs). In particular, the authors obtain the variance of the limiting random term in the reconstruction but do not establish its Gaussianity. However, the assumption made here is again that of  oversampled data, even though the assumption of i.i.d. noise has been avoided. [\textcolor{red}{Sasha, can you check this: especially Hypothesis 4.2}]. As opposed to these works, we emphasize that due to its local nature, our method is especially useful for {analyzing} reconstructed images of objects with sharp boundaries (which could be smooth \cite{Katsevich2017a, kat19a, Katsevich2020a, Katsevich2020b, Katsevich2021a} or highly irregular \cite{Katsevich2022a, Katsevich2023a, Katsevich_2024_BV}). 
%
%
%In this work we aim to close the gap by considering the case when Radon transform data is corrupted by noise with an unknown distribution function. In fact,  for \textit{local reconstruction analysis}, we avoid the twin assumptions of oversampling as well as i.i.d. noise, thus making this analysis more suitable for high-fidelity applications such as medical diagnostics. 

The paper is organized as follows. In section \ref{sec:setting_mainres}, we give a mathematical formulation of our problems and state the main results. In section \ref{sec:main proofs}, we state and prove our main theorems while deferring some key technical results to section \ref{sec: 2nd DTB}. In section \ref{sec:numerics}, we validate our main results through simulated numerical experiments. Finally, we collect the proofs of some auxiliary lemmas in Appendices \ref{approx_step}, \ref{sec:first two lems}, and \ref{sec:two lems}.

\section{Setting of the problem and main results.}\label{sec:setting_mainres}
We now describe the problem of reconstructing a function $f(x)$, $x\in \Rb^2$, from discretely sampled noisy Radon Transform  (RT) data. Let us first define the parameters that we use to discretize the observation space, $\Sb^1\times [-P,P]$. To this end, let:
\be\label{params}
\al_k=k\Delta\al,\ p_j=\bar p+j\Delta p,\ \Delta p=\e,\ \Delta\al/\Delta p=\kappa,
\ee 
where $\kappa>0$ and $\bar p$ are fixed.
We parametrize $\vec\alpha_k\in \Sb^1$ by $\vec\alpha_k=(\cos\alpha_k,\sin\alpha_k)$. Similarly, the radial (signed) distance is discretized as $p_j=\bar{p}+j\Delta p$. We will loosely refer to $\epsilon$ as the data step-size. The discrete noisy tomographic data is modeled as:
\be\label{disc_model}
\hat{f}_{\e,\eta}(\al_k,p_j)= Rf(\al_k,p_j)+ \eta_{k,j},
\ee
where $Rf(\al_k,p_j)$ is the Radon transform of the function at the grid point $(\al_k,p_j)$ in the observation space and $\eta_{k,j}:=\eta(\al_k,p_j)$ are random variables that model noise in the observed data. We assume $\eta_{k,j}$ are independent but not necessarily identically distributed. We make the following assumptions on the first three moments of the random variables $\eta_{k,j}$. Here and below, $\Eb(X)$ denotes the expected value of a random variable $X$.

For convenience, throughout the paper we use the following convention. If a constant $c$ is used in an equation or an inequality, the qualifier ‘for some $c>0$’ is assumed. If several $c$’s are used in a string of (in)equalities, then ‘for some’ applies to each of them, and the values of different $c$’s may all be different. For example, in the string of inequalities $f\le cg \le ch$, the values of $c>0$ in two places may be different.

We now state our main assumptions on the measurement noise.

\begin{assumption} {(Assumptions on noise)}\label{noi}
\begin{enumerate}
\item $\Eb(\eta_{k,j})=0$.
\item $\Eb\eta_{k,j}^2=\sigma^2(\al_k,p_j)\Delta\al$ for some  $\sigma\in C^1([-\pi,\pi]\times[-P,P])$.
\item $\Eb \lvert \eta_{k,j}\rvert^3\leq c(\Delta \alpha)^{3/2}$.
\end{enumerate}
\end{assumption}

We also select an interpolating kernel, $\ik$.
\begin{assumption}{(Assumptions on the kernel $\ik$)}\label{interp}
\begin{enumerate}
    \item $\ik$ is compactly supported.
    \item $\ik^{(M+1)}\in L^\infty(\br)$ for some $M\ge 3$.
    \item $\int \ik(t)dt=1$.
\end{enumerate}
\end{assumption}

Usually we make an additional assumption that $\ik$ exactly interpolates polynomials up to some degree $m_{\max}$ \cite{Katsevich2017a, kat19a, Katsevich2020a, Katsevich2020b, Katsevich2021a}: 
\be\label{interp cond}
\sum_{j\in \Zb}j^m\ik(t-j)=t^m,\ t\in \Rb,\ m=0,1,\dots,m_{\max}.
\ee
Here this assumption is not needed. In particular, $\ik$ may account for the effects of smoothing that can be used to reduce noise in the reconstruction. In this case, $\ik$ no longer satisfies \eqref{interp cond}.

Denoting the Hilbert transform of a function by $\Hc(\cdot)$, the reconstruction formula from the data \eqref{disc_model} is given by:
\begin{align}\label{recon-orig}
f^{\text{rec}}_{\e,\eta}(x)&=-\frac{\Delta\al}{4\pi\e}\sideset{}{_{|\al_k|\le \pi}}\sum \sideset{}{_j}\sum\CH \ik^{\prime}\left(\frac{\vec\al_k\cdot x-p_j}\e\right)\hat f_{\e,\eta}(\al_k,p_j)\nonumber\\
&= f^{\text{rec}}_{\e}(x)+N^{\text{rec}}_{\e}(x),
\end{align}
where we define
\be\label{fN recons}\begin{split}
   f^{\text{rec}}_{\e}(x)&  :=-\frac{\Delta\al}{4\pi\e}\sideset{}{_{|\al_k|\le \pi}}\sum \sideset{}{_j}\sum\CH \ik^{\prime}\left(\frac{\vec\al_k\cdot x-p_j}\e\right) Rf(\al_k,p_j),\\
   N^{\text{rec}}_{\e}(x)&:=-\frac{\Delta\al}{4\pi\e}\sideset{}{_{|\al_k|\le \pi}}\sum \sideset{}{_j}\sum\CH \ik^{\prime}\left(\frac{\vec\al_k\cdot x-p_j}\e\right)\eta_{k,j}.
\end{split}
\ee 
Since $\ik$ is compactly supported, we see from \eqref{fN recons} that the resolution of the reconstruction is, roughly, of order $\sim\e$, i.e. of the same order as the data step-size. The asymptotic behaviour of $f^{\text{rec}}_{\e}(x)$ as the data step-size becomes vanishingly small, i.e., $\lim_{\e\to 0} f^{\text{rec}}_{\e}(x)$ is well-understood from the theory of local reconstruction analysis (LRA), see e.g. \cite{Katsevich2017a, kat19a, Katsevich2020a, Katsevich2020b, Katsevich2021a}.  In the spirit of LRA, we seek to approximate $N_\e(x_0+\e\check x)$, where $x_0$ is fixed, $\chx$ is restricted to a bounded set, and
\be
N_\e^{\text{rec}}(x_0+\e\chx)=-\frac{\Delta\al}{4\pi\e}\sum_{j,k} \Hc\ik^{\prime}\big(a_k-j+\vec{\alpha}_k\cdot \chx \big)\eta_{k,j},\ a_k:=(\Ab_k\cdot x_0-\bar p)/\e.\label{recon}
\ee
To this end, we first establish that 
\be\label{GRF}
N^{\text{rec}}(\chx):=\lim_{\e\to0} \Ner(x_0+\e\chx)
\ee
is a Gaussian random variable for any fixed $\chx $. We will generalize this result further to conclude that as $\chx $ varies in a \textit{neighborhood} of a generic $x_0$ it gives rise to a Gaussian random field (GRF). By a slight abuse of notation, the latter is also denoted by $N^{\text{rec}}(\chx)$.  

Now we state a key technical assumption on the center of any neighborhood of $x_0$ that is needed later to state our main theorems. Let $\langle r\rangle$ denote the distance from a real number $r\in\br$ to the integers, $\langle r\rangle:=\text{dist}(r,\mathbb Z)$. The following definition is in \cite[p. 121]{KN_06} (after a slight modification in the spirit of \cite[p. 172]{Naito2004}).

\begin{definition} Let $\nu>0$. The irrational number $s$ is said to be of type $\nu$ if for any $\nu_1>\nu$, there exists $c(s,\nu_1)>0$ such that
\be\label{type ineq}
m^{\nu_1}\langle ms\rangle \ge c(s,\nu_1) \text{ for any } m\in\mathbb N.
\ee
\end{definition}
See also \cite{Naito2004}, where the numbers which satisfy \eqref{type ineq} are called $(\nu-1)$-order Roth numbers. It is known that $\nu\ge1$ for any irrational $s$. The set of irrationals of each type $\nu \ge 1$ is of full measure in the Lebesgue sense \cite{Naito2004}.

\begin{assumption}{(Assumptions on the center of a neighborhood {of} $x_0$)}\label{ass:x0}
\begin{enumerate}
    \item The quantity $\kappa|x_0|$ is irrational and of some finite type $\nu$.
    \item $\sigma^2(\al,\vec\al\cdot x_0)\not=0$ for all $\al$ in some open set $\Omega\subset[0,2\pi]$.
\end{enumerate}
\end{assumption}

\noindent
Now we are ready to state the main theorems proved in this work.
\begin{theorem}\label{lem:Lyapunov}
Let $x_0,\chx \in\Rb^2$ be two fixed points. Suppose the random variables $\eta_{k,j}$ satisfy Assumption \ref{noi}, the kernel $\ik$ satisfies Assumption \ref{interp} with $M>\nu+1$, and the point $x_0$ satisfies Assumption~\ref{ass:x0}. One has
\begin{align}\label{main lim}
\frac{\sum_{j,k}\lvert \Hc \ik^{\prime}(a_k-j,\Ab_k\cdot \chx)\rvert^3\Eb\lvert \eta_{j,k}\rvert^3}{\big[\sum_{k,j}\big(\Hc \ik^{\prime}(a_k-j,\Ab_k\cdot \chx)\big)^2\Eb \eta_{k,j}^2\big]^{\frac{3}{2}}}=O(\e^{1/2}),\ \e\to0.
\end{align}
\end{theorem}

\begin{corollary}\label{cor:var}
The family of random variables ${N}_\e^{\text{rec}}(\chx)$, $\e>0$, satisfies the Lyapunov  condition for triangular arrays \cite [Definition 11.1.3]{ath_book}. By \cite[Corollary 11.1.4]{ath_book}, $N^{\text{rec}}(\chx):=\lim_{\e\to0}{N}_\e^{\text{rec}}(\chx)$ is a Gaussian random variable, where the limit is in the sense of convergence in distribution.
\end{corollary}

Our next theorem shows that if we consider the reconstruction $N^{\text{rec}}_{\e}(x)$ at any finite number of fixed points in a neighborhood of some chosen point $x_0$, then, in the limit as $\e\to0$, the reconstruction is a Gaussian random vector. More precisely, let us select any $K$ distinct points $\chx_i\in\br^2$, $i=1,\dots,K$. The corresponding  reconstruction vector is ${\mathbf{N}}_\e^{\text{rec}}:=({N}_\e^{\text{rec}}(x_0+\e\chx_1),\dots,{N}_\e^{\text{rec}}(x_0+\e\chx_K))\in\br^K$. Pick any vector $\vec\theta\in \mathbb R^K$. By \eqref{recon}
\begin{align}\label{recon dttpr 1}
\xi_\e:=\vec\theta\cdot \mathbf{N}_\e^{\text{rec}} = \sum_{i=1}^K \theta_i\sum_{j,k} \Hc\ik^{\prime}\big(a_k-j+\vec{\alpha}_k\cdot \chx_i\big)\eta_{k,j}.
\end{align}
The next theorem generalizes Theorem \ref{lem:Lyapunov} above.

\begin{theorem}\label{thm:Lyapunov2} 
Pick any $\vec\theta\in \Rb^K$, $\vec\theta\not=0$, and let $\xi_{\e}$ be defined as in \eqref{recon dttpr 1}. Suppose the random variables $\eta_{k,j}$ satisfy Assumption \ref{noi}, the kernel $\ik$ satisfies Assumption \ref{interp} with $M>\nu+1$, and the point $x_0$ satisfies Assumption~\ref{ass:x0}. One has:
\begin{align}\label{eq:6.7}
\frac{\sum_{j,k}\lvert \sum_{i=1}^K \theta_i\Hc \ik^{\prime}(a_k-j,\Ab_k\cdot \chx_i)\rvert^3\Eb\lvert \eta_{j,k}\rvert^3}{\big[\sum_{k,j}\big(\sum_{i=1}^K \theta_i\Hc \ik^{\prime}(a_k-j,\Ab_k\cdot \chx_i)\big)^2\Eb \eta_{k,j}^2\big]^{\frac{3}{2}}}=O(\e^{1/2}),\ \e\to0.
\end{align}
\end{theorem}

\begin{corollary}\label{cor:covar}
From \cite [Corollary 11.1.4]{ath_book}, it follows that $\lim_{\e\to0} \xi_{\e}$ is a Gaussian random variable. Hence, by \cite [Theorem 10.4.5]{ath_book} $\lim_{\e\to0}{ \mathbf{N}}_\e^{\text{rec}}$ is a Gaussian random vector, where as before, the limit is in the sense of convergence in distribution. 
\end{corollary}

Let $D\subset\br^2$ be a domain. Recall that $G(x)$, $x\in D$, is a Gaussian random field (GRF) if $(G(x_1),\cdots,G(x_K))$ is a Gaussian random vector for any $K\ge1$ and any collection of points $x_1,\cdots,x_K\in D$ \cite[Section 1.7]{AdlerGeomRF2010}. As is known, a GRF is completely characterized by its mean function $m(x)=\Eb G(x)$, $x\in D$ and its covariance function $\text{Cov}(x,y)=\Eb (G(x)-m(x))(G(y)-m(y))$, $x,y\in D$ \cite[Section 1.7]{AdlerGeomRF2010}. Thus, Corollary~\ref{cor:covar} implies that $N^{\text{rec}}(\chx)$ is a GRF.

Let $D:=[A_1,A_2]\times[B_1,B_2]$ be a rectangle. In the next theorem, we show that $N_\e^{\text{rec}}(x_0+\e\check x)\to N^{\text{rec}}(\chx)$, $\chx \in D$, as $\e\to0$ weakly (\cite[p. 185]{Khoshnevisan2002}). Recall that $N^{\text{rec}}(\chx)$, $\chx \in D$, denotes a GRF as well (i.e., not just a random variable). Given two compactly supported, real-valued continuous functions $f$ and $g$, their cross-correlation is defined as follows:
\be\label{cross-cor}
(f\star g)(t):=\int_\br f(t+s)g(s)\dd s.
\ee

\begin{theorem}\label{GRF_thm}
Let $D$ be a rectangle. Suppose the random variables $\eta_{k,j}$ satisfy Assumption \ref{noi}, the kernel $\ik$ satisfies Assumption \ref{interp} with $M>\max(\nu+1,3)$, and the point $x_0$ satisfies Assumption~\ref{ass:x0}.
Then, $N_{\e}^{\text{rec}}(x_0+\e \chx)\to N^{\text{rec}}(\chx)$, $\chx\in D$, $\e\to0$, as GRFs in the sense of weak convergence. Furthermore, $N^{\text{rec}}(\chx)$ is a GRF with zero mean and covariance
\be\label{Cov main}
\text{Cov}(\chx,\chy)
=C(\chx-\chy):=\bigg(\frac{\kappa}{4\pi}\bigg)^2\int_{0}^{2\pi}\sigma^{2}(\alpha,\Ab\cdot x_0) (\phi^{\prime}\star \phi^{\prime})(\Ab\cdot (\chx-\chy))\dd\alpha,
\ee
and sample paths of $N^{\text{rec}}(\chx)$ are continuous with probability $1$.
\end{theorem}

\section{Proofs of Theorems~\ref{lem:Lyapunov}--\ref{GRF_thm}}\label{sec:main proofs}

\subsection{Proof of Theorem \ref{lem:Lyapunov}}
Similarly to \cite{Katsevich_2024_BV}, we define 
\begin{equation}\label{psi_a_b}\begin{split}
\psi(a,b):=&\sum_j  \big[ \Hc \ik^{\prime}(a-j+b) \big]^2.
\end{split}
\end{equation} 
Since $\Hc \ik^{\prime}(t)=O(t^{-2})$, $t\to\infty$, the series above converges absolutely. 
It is easy to see that 
\begin{align}\label{Psi per}
\psi(a+1,b)=  \psi(a,b),\ \forall a,b\in\mathbb R.
\end{align}

We analyze the numerator and denominator in \eqref{main lim} separately. It is shown in Appendix~\ref{approx_step} that the denominator in \eqref{main lim} can be written as $d_\e^{3/2}$, where
\begin{equation}\label{Deps}
d_\e:=\Delta\alpha\sum_{|\al_k|\le \pi} \psi(a_k,\vec{\alpha}_k\cdot \chx)\sigma^2(\alpha_k,\vec\alpha_k\cdot x_0)+O(\e).
\end{equation}
The leading term in \eqref{Deps} is obtained by substituting $p_j=\vec\alpha_k\cdot x_0$ in the second argument of $\sigma^2(\al_k,p_j)=\Eb \eta_{k,j}^2$.

Next we want to evaluate the limit, $\lim_{\epsilon \to 0}d_\e$. Using arguments similar to \cite{Katsevich_2024_BV}, we prove in Section~\ref{sec: 2nd DTB} the following result
\begin{align}\label{lim De}
\lim_{\epsilon \to 0}d_\e&=\int_{0}^{2\pi}\sigma^2(\alpha,\vec\alpha\cdot x_0)\int_0^1 \psi(r,\Ab\cdot \chx) \dd r \dd\alpha.
\end{align} 
Here 
\begin{equation}\label{lim psi int}
\int_{0}^1 \psi(r,b)\dd r=\sum_{j} \int_{0}^1\left( \Hc \ik^{\prime}(r-j+b)\right)^2\dd r=\int_{\Rb}\left(\Hc \ik^{\prime}(r)\right)^2\dd r=:C>0.
\end{equation}
Using Parseval's theorem, we also have:
\begin{align}
    \int_{\Rb}\left(\Hc \ik^{\prime}(r)\right)^2\dd r =  (2\pi)^{-1}\int_{\Rb}\big| \lambda \tilde{\ik}(\lambda)\big|^2\dd\lambda,
\end{align}
where $\tilde{\ik}(\lambda)$ denotes the Fourier transform of $\ik$. Thus 
\begin{align}\label{A.6}
\lim_{\epsilon \to 0}d_\e=C\int_0^{2\pi} \sigma^2(\alpha,\vec\alpha\cdot x_0) \dd\alpha>0
\end{align}
due to assumption~\ref{ass:x0}(2). To study the numerator in \eqref{main lim}, we define similarly to \eqref{psi_a_b}
\be \label{Psi def}
\Psi(a,b):=\sum_j  \big| \Hc \ik^{\prime}(a-j+b)\big|^3.
\ee
Clearly, $|\Psi(a,b)|\le c<\infty$, $a,b\in\br$. The numerator in \eqref{main lim} is bounded by $c\e^{1/2}n_\e$, where
\begin{align}\label{A.7}
n_\e:= \Delta\alpha\sum_{\lvert\al_k\rvert\leq \pi}\Psi\bigl(a_k,\Ab_k\cdot \check x\bigr).
\end{align}
Hence $n_\e\le c<\infty$ for all $0<\e<1$. Combining this with eq. \eqref{A.6} proves the theorem.

\subsection{Proof of Theorem \ref{thm:Lyapunov2} and Corollary~\ref{cor:var}}
{To show that ${\mathbf{N}}_\e^{\text{rec}}$ converges in distribution to a Gaussian random vector, it suffices to show that for any $0\neq \vec\theta\in \Rb^K$, $\lim_{\e\to0}\xi_{\e}:=\lim_{\e\to 0}\vec\theta\cdot \mathbf{N}_\e^{\text{rec}}$ is a Gaussian random variable \cite [Theorem 10.4.5]{ath_book}. Thus if we establish \eqref{eq:6.7}, then from Lyapunov's CLT, we will have shown $\xi_{\e}$ converges (in distribution) to a Gaussian random variable and consequently, ${\mathbf{N}}_\e^{\text{rec}}$ converges to a Gaussian random vector as $\e\to0$. The proof of this claim is similar to that of Theorem~\ref{lem:Lyapunov}, so here we only highlight the key points.

First we show that the denominator in \eqref{eq:6.7} converges to a positive number. 
Similarly to \eqref{Deps}, we show in Appendix \ref{approx_step}} that
\begin{equation}\label{recon dttpr 2}
\begin{split}
\Eb\xi_\e^2=\Delta\alpha\sum_{i_1=1}^K\sum_{i_2=1}^K  \theta_{i_1}\theta_{i_2}\sum_{j,k}& \Hc\ik^{\prime}\big(a_k-j+\vec{\alpha}_k\cdot \chx_{i_1}\big)\Hc\ik^{\prime}\big(a_k-j+\vec{\alpha}_k\cdot \chx_{i_2}\big)\\
&\times \sigma^2(\alpha_k,\vec\alpha_k\cdot x_0) + O(\e).
\end{split}
\end{equation} 
Therefore, using the same arguments as in the proof of Theorem~\ref{lem:Lyapunov}, we obtain
\begin{equation}\label{recon dttpr 3}
\begin{split}
\lim_{\e\to0}\Eb\xi_\e^2=&\int_0^{2\pi}\int_\br \sum_{i_1=1}^K\sum_{i_2=1}^K  \theta_{i_1}\theta_{i_2}\Hc\ik^{\prime}(r+\vec{\alpha}\cdot \chx_{i_1})\\
&\times \Hc\ik^{\prime}(r+\vec{\alpha}\cdot \chx_{i_2})\sigma^2(\alpha,\vec\alpha\cdot x_0)\dd r\dd\alpha
=\int_0^{2\pi}\sigma^2(\alpha,\vec\alpha\cdot x_0)\int_\br f^2(r,\alpha)\dd r\dd\alpha,\\
f(r,\alpha):=&\sum_{i=1}^K  \theta_i\Hc\ik^{\prime}(r+\vec{\alpha}\cdot \chx_i).
\end{split}
\end{equation} 

Suppose the limit in \eqref{recon dttpr 3} is zero. Clearly, $f(r,\al)$ is analytic in $r\in\mathbb C$ outside a compact subset of the real line for any $\al$. By assumption~\ref{ass:x0}(2), $f(r,\alpha)\equiv0$, $r\in\br$, $\al\in\Omega$. By analytic continuation and the Sokhotski–Plemelj formulas \cite[Chapter 1, section 4.2]{gakhov}, 
\be\label{zero sum}
\sum_{i=1}^K  \theta_i\ik^{\prime}(p+\vec{\alpha}\cdot \chx_i)\equiv0,\ p\in\br,\al\in\Omega.
\ee
Recall that all $\chx_i$ are distinct. Since $\Omega$ is an open set, we can find $\alpha\in\Omega$ such that $\vec\alpha \cdot \chx_{i_1}\not=\vec\alpha \cdot \chx_{i_2}$, $i_1,i_2=1,2,\dots,K$, $i_1\not=i_2$. This can be done by finding a plane $\vec\al^\perp$ through the origin that does not contain any of the vectors $\chx_{i_1}-\chx_{i_2}$, $i_1,i_2=1,2,\dots,K$, $i_1\not=i_2$. Together with \eqref{zero sum} this easily implies that all $\theta_i$ are zero. Since we assumed that $\vec\theta\not=0$, this contradiction proves that the limit in \eqref{recon dttpr 3} is not zero.

Finally, we analyze the numerator in \eqref{eq:6.7}. Using Assumption \ref{noi}(3) and arguing similarly to \eqref{Psi def}, \eqref{A.7}, we obtain that the numerator is bounded above as,  $\Eb\lvert\xi_\e\rvert^3\leq c\epsilon^{1/2}$. This finishes the proof.

\subsection{Proof of Theorem \ref{GRF_thm}}
Define $C:=C(D,\br)$ to be the collection of all continuous functions $f:D\to\br$ metrized by
\be\label{Cmetric}
d(f,g)=\sup_{\chx\in D}|f(\chx)-g(\chx)|,\ f,g\in C.
\ee
Our goal is to show that $\Ner(x_0+\e\chx)$, $\chx\in D$, converges to $N^{\text{rec}}(\chx)$, $\chx\in D$, in distribution as $C$-valued random variables. We use the following definition and theorem.

\begin{definition}[{\cite[p. 189]{Khoshnevisan2002}}]\label{def:kh}
Let $\Pb_n$ be the distribution of a $C$-valued random variable $X_n$, $1\leq n\leq \infty$. The collection $(\Pb_n)$ is tight if for all $\de\in(0,1)$, there exists a compact set $\Gamma_\de\in C$ such that $\sup_n \Pb(X_n\not\in\Gamma_\de)\le \de$. 
\end{definition}

\begin{theorem}[{\cite[Proposition 3.3.1]{Khoshnevisan2002}}]\label{kh_thm}
Suppose $X_n,\ 1\leq n\leq \infty$, are $C$-valued random variables. Then $X_n\to X_{\infty}$ weakly (i.e., the distribution of $X_n$ converges to that of $X_\infty$, see {\cite[p. 185]{Khoshnevisan2002}}) provided that:
\begin{enumerate}
\item Finite dimensional distributions of $X_n$ converge to that of $X_{\infty}$.
\item $(X_n)$ is a tight sequence.
\end{enumerate}
\end{theorem}

Corollary \ref{cor:covar} asserts that all finite-dimension distributions of $\Ner(x_0+\e\chx)$ converge to that of $N^{\text{rec}}(\chx)$. Thus what remains to be verified is Property 2 of Theorem \ref{kh_thm}. To this end, we
consider the sets 
\be\label{compact sets}
\Gamma_\de:=\{f\in C:\, \Vert f\Vert_{W^{2,2}(D^\circ)}^2 \le 1/\de\},
\ee
where $D^\circ$ is the interior of $D$. Recall that $W^{k,p}(D^\circ)$ is the closure of $C^\infty(D)$ in the norm:
\be\label{Wkp norm}
\Vert f\Vert_{k,p}:=\bigg(\int_{D}\sum_{|m|\le k}|\pa_x f(x)|^p\dd x\bigg)^{1/p},\ f\in C^\infty(D).
\ee

By \cite[eq. (7.8), p. 146 and Theorem 7.26, p. 171]{GilbTrud}, the imbedding $W^{2,2}(D^\circ)\hookrightarrow C(D)$ is compact. More precisely, we use here that the imbedding $W^{2,2}(D^\circ)\hookrightarrow W^{2,p}(D^\circ)$, $1\le p\le 2$ is continuous \cite[eq. (7.8), p. 146]{GilbTrud}, and the imbedding $W^{2,p}(D^\circ)\hookrightarrow C(D)$, $1\le p< 2$, is compact \cite[Theorem 7.26, p. 171]{GilbTrud}. Recall that $D$ is a rectangle, so its boundary is Lipschitz continuous. Hence the set $\Gamma_\de\subset C$ is compact for every $\de>0$. From \eqref{recon},
\be\label{recon deriv}
\pa_{\chx}^m N_\e^{\text{rec}}(x_0+\e\chx)=c\sum_{j,k}\vec\al_k^m \Hc\ik^{(|m|+1)}\big(a_k-j+\vec{\alpha}_k\cdot \chx \big)\eta_{k,j},\
m\in\N_0^2,|m|\le2.
\ee
Recall that $a_k$ are defined in \eqref{recon dttpr 2}. Therefore
\be\label{var der}
\begin{split}
\Eb(\pa_{\chx}^m N_\e^{\text{rec}}(x_0+\e\chx))^2=&c \Delta\al\sum_{j,k} \big[\vec\al_k^m\Hc\ik^{(|m|+1)}\big(a_k-j+\vec{\alpha}_k\cdot \chx \big)\big]^2\sigma^2(\al_k,p_j)\\
\le & c,\  \chx\in D.
\end{split}
\ee
This implies that $\Eb \Vert N_\e^{\text{rec}}(x_0+\e\chx))\Vert_{W^{2,2}(D)}^2\le c$ for all $\e>0$.
By the Chebyshev inequality,
\be\label{cheb}
\BP(N_\e^{\text{rec}}(x_0+\e\chx)\not\in\Gamma_\de)=\BP(\Vert N_\e^{\text{rec}}(x_0+\e\chx)\Vert_{W^{2,2}(D)}^2 \ge 1/\de)\le c\de.
\ee
Therefore $(X_n)$ is a tight sequence. 

By Theorem~\ref{kh_thm}, $\Ner(x_0+\e\chx)\to N^{\text{rec}}(\chx)$ in distribution as $C$-valued random variables. Since $C$ is a complete metric space, it follows that $N^{\text{rec}}(\chx)$ has continuous sample paths with probability 1.

By the linearity of the expectation, $N^{\text{rec}}(\chx)$ is a zero mean GRF. To completely characterize this GRF, we calculate its covariance function $\text{Cov}(\chx,\chy)=\Eb(N_\e^{\text{rec}}(\chx)N_\e^{\text{rec}}(\chy))$,  $\chx,\chy\in D$. In fact, essentially this has already been done in the proof of Theorem~\ref{thm:Lyapunov2}. From \eqref{recon} and \eqref{recon dttpr 3} we obtain
\begin{align}
\label{C_lm}
\begin{split}
\text{Cov}(\chx,\chy)&=\bigg(\frac{\kappa}{4\pi}\bigg)^2\int_{0}^{2\pi}\sigma^{2}(\alpha,\Ab\cdot x_0)\int_\br \Hc \phi^{\prime}(r+\Ab \cdot \chx)\Hc \phi^{\prime}(r+\Ab\cdot \chy) \dd r \dd\alpha.
\end{split}
\end{align}
The integral with respect to $r$ in \eqref{C_lm} simplifies as follows:
\begin{equation}
\label{int_simplify}
\begin{split}
\int_{\mathbb{R}}
\Hc \phi^{\prime}(r+\Ab\cdot \chx)  \Hc \phi^{\prime}(r+\Ab\cdot \chy) \dd r
&= \left(\Hc \phi^{\prime}\star \Hc \phi^{\prime}\right)\left(\Ab\cdot (\chx - \chy)\right)\\
&= \left(\phi^{\prime}\star \phi^{\prime}\right)\left(\Ab\cdot (\chx - \chy)\right),
\end{split}
\end{equation}
and \eqref{Cov main} is proven. The last step of \eqref{int_simplify} follows since $\mathcal{F}\left(\Hc \phi^{\prime}\right)(\xi) = -i\text{sgn}(\xi)\mathcal{F}(\phi^{\prime})(\xi)$, and by the convolution theorem.

\section{Proof of \eqref{lim De}}\label{sec: 2nd DTB}

We can write \eqref{Deps} in the following form
\begin{equation}\label{Deps 1}
d_\e=\Delta\al\sum_{|\al_k|\le \pi} g(a_k,\alpha_k)+O(\e),\quad
g(r,\al):=\psi(r,\vec{\alpha}\cdot \check x)\sigma^2(\alpha,\vec\al\cdot x_0).
\end{equation}
By \eqref{Psi per}, $g(r,\al)=g(r+1,\al)$ for any $r$ and $\al$. Represent $g$ in terms of its Fourier series:
\be\label{four-ser}\begin{split}
g(r,\al)&=\sum_{m\in\mathbb Z} \tilde g_m(\al) e(-mr),\ e(r):=\exp(2\pi i r),\\
\tilde g_m(\al) &=\int_0^1 g(r,\al)e(mr)\dd r=\sigma^2(\alpha,\vec\al\cdot x_0)\int_\br \left[ \Hc \ik^{\prime}(r+\vec{\alpha}\cdot \check x)\right]^2 e(mr)\dd r\\
&=\sigma^2(\alpha,\vec\al\cdot x_0)e(-m\vec{\alpha}\cdot \check x)\tilde{\tilde g}_m,\
\tilde{\tilde g}_m:=\int_\br \left[ \Hc \ik^{\prime}(r)\right]^2 e(mr)\dd r.
\end{split}
\ee
Let us introduce the function $\rho(s):=(1+|s|)^{-M}$, $s\in\br$, where $M$ is the same as in Assumption~\ref{interp}(2). Then we have the following lemma.
\begin{lemma}\label{lem:psi psider}
One has
\be\label{four-coef-bnd}
|\tilde{\tilde g}_m| \le c\rho(m),\ |\tilde g_m^{\prime}(\al)|\le c_k (1+|m|)\rho(m),\ |\al|\le\pi,\ m\in\mathbb Z.
\ee
\end{lemma}
The proof is immediate using assumption~\ref{interp}(2) and integrating by parts $M$ times.
By the last lemma, the Fourier series for $g$ converges absolutely. From \eqref{Deps 1} and \eqref{four-ser}, 
\be\label{recon-ker-v2}
\begin{split}
d_\e&={\Delta\al}\sum_{m\in\mathbb Z} \sum_{|\al_k|\le \pi} e\left(-m \frac{\vec\al_k\cdot x_0-\bar p}\e\right)\tilde g_m(\al_k)+O(\e).
\end{split}
\ee
To prove \eqref{lim De}, it suffices to prove the following two statements:
\begin{align}\label{extra terms I}
&\e\sum_{m\not=0}\bigg| \sum_{|\al_k|\le \pi} e\left(-m \frac{\vec\al_k\cdot x_0}\e\right) \tilde g_m(\al_k)\biggr|=O(\e^{1/2}),\\
&\label{extra terms II}
\sum_{|\al_k|\le\pi} \int_{|\al-\al_k|\le \Delta\al/2}|\tilde g_0(\al)-\tilde g_0(\al_k)|\dd\al=O(\e),\quad \e\to 0.
\end{align}

For a compact $I$ and a $C^k(I)$ function $\phi$ define
\be\label{mx}
\phi_{\text{mx}}^{(k)}(I):=\max_{\al\in I}|\phi^{(k)}(\al)|,\ k\ge 0.
\ee
Set $\ga=1/(2(M-2))$. Since $\sum_{|m|\ge \e^{-\ga}}\rho(m)=O(\e^{1/2})$, in \eqref{extra terms I} we can restrict $m$ to the range $1\le|m|\le\e^{-\ga}$. The result \eqref{extra terms II} is obvious, because $|\tilde g_0^{\prime}(\al)|$ is bounded on $[-\pi,\pi]$.

The following lemma is proven in \cite{Katsevich_2024_BV}.

\begin{lemma}\label{lem:sum1 est}
Let $I$ be an interval. Pick two functions $\phi$ and $g$ such that $\phi\in C^2(I)$,  $\phi_{\text{mx}}^{\prime\prime}(I)<\infty$; and $g\in C^1(I)$. 
Suppose that for some $l\in \mathbb Z$ one has
\be\label{Omega_eps}
|\kappa  \phi^{\prime}(\al)-l|\le 1/2\text{ for any }\al\in I.
\ee
Denote 
\be\label{h-fn}\begin{split}
&\fks(\al):=\phi(\al)-(l/\kappa)\al,\\
&h_l(\al):=\frac{\pi\kappa \phi^{\prime}(\al)}{\sin(\pi\kappa \phi^{\prime}(\al))}\text{ if } \phi^{\prime}(\al)\not=0,
\end{split}
\ee
and $h_l(\al):=1$ if $\phi^{\prime}(\al)=0$. For all $\e>0$ sufficiently small, one has
\be\label{sum est}\begin{split}
&\biggl|\Delta\al\sum_{\al_k\in I}g(\al_k)e\biggl(\frac{\phi(\al_k)}\e\biggr)-\int_{I} g(\al)h_l(\al) e\biggl(\frac{\fks(\al)}\e\biggr)\dd \al\biggr|\\
&\le c\e\biggl[\bigl(1+\e\phi_{\text{mx}}^{\prime\prime}(I)\bigr)\int_I|g^{\prime}(\al)|\dd\al+\phi_{\text{mx}}^{\prime\prime}(I)\int_I|g(\al)|\dd\al \biggr],
\end{split}
\ee
where the constant $c$ is independent of $\e$, $\phi$,  $g$, and $I$.
\end{lemma}

\noindent The following lemma is proven in section~\ref{ssec:prf int1 est}.

\begin{lemma}\label{lem:int1 est} Pick any interval $I=[a,b]$ and a function $g\in C^1(I)$. Set $\phi(\al):=-m \vec\al\cdot x_0$, $m\not=0$. Suppose 
\begin{enumerate}
\item $\phi^{\prime\prime}(\al)\not=0$, $\al\in(a,b)$.
\item There exists an integer, $l$, such that $\kappa  \phi^{\prime}(\al_0)=l$ for some $\al_0\in I$ and $  \phi^{\prime\prime}(\al_0)\not=0$.
\item $|\kappa  \phi^{\prime}(\al)-l|\le 1/2$ for any $\al\in I$.
\end{enumerate}
One has
\be\label{int1 est}\begin{split}
&\biggl|\int_{I} g(\al) h_l(\al) e\biggl(\frac{\fks(\al)}\e\biggr)\dd \al\biggr|
\le c \e^{1/2}\bigg[g_{\text{mx}}(I)+\int_I|g^{\prime}(\al)|\dd\al\biggr]\big(\phi^{\prime\prime}(\al_0)\big)^{-1/2},
\end{split}
\ee
where the constant $c$ is independent of $g$, $m$, $\e$, $l$ and $I$.
\end{lemma}

In what follows we use Lemmas~\ref{lem:sum1 est} and \ref{lem:int1 est} with $g(\al):=\tilde g_m(\al)$. Recall that $\phi(\al)=-m \vec\al\cdot x_0$ and $m\not=0$. Let $\alst\in(-\pi,\pi]$ be the global maximum of $\phi^{\prime}$, so $\phi^{\prime\prime}(\alst)=0$ and $\vec\al_\star\cdot x_0=0$, see Figure~\ref{fig:sine}. Clearly, $\phi(\al)=-m|x_0| \cos(\al-\al_{x_0})$ and $\phi^{\prime}(\al)=m|x_0| \sin(\al-\al_{x_0})$. Hence we can write 
\be\label{phi pr alt}
\kappa\phi^{\prime}(\al)=\cx m\cos(\al-\alst),\ \cx:=\kappa|x_0|\text{sgn}(m),\ |\al|\le\pi. 
\ee

\begin{figure}[h]
{\centerline{
{\epsfig{file={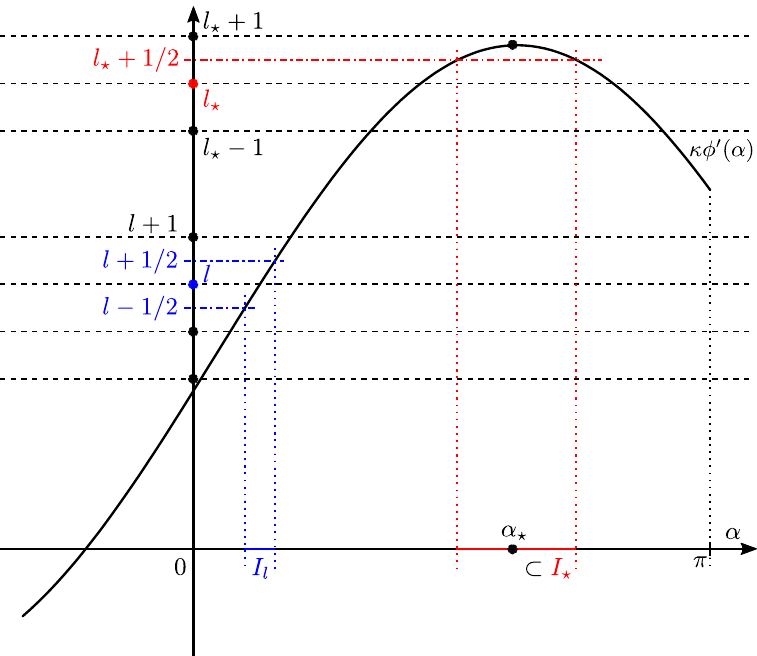}, width=10cm}}
}}
\caption{Illustration of the quantity $l_\star$, the set $I_l$, \and one of the two intervals that make up $I_\star$, namely the interval $\{|\al|\le\pi: \kappa\phi^{\prime}(\al)\ge l_\star+(1/2)\}$. The other interval $\{|\al|\le\pi: \kappa\phi^{\prime}(\al)\le -(l_\star+(1/2))\}$ is not visible.}
\label{fig:sine}
\end{figure}

Given $r\in\br$, let $\lfloor r\rfloor$ denote the floor function, i.e. the integer $n\in\mathbb Z$ such that $n\le r<n+1$. Denote (Figure~\ref{fig:sine})
\be\label{special l}
l_\star(m):=\lfloor\kappa \phi^{\prime}(\alst)\rfloor,\ m\not=0.
\ee
Note that $\kappa\phi^{\prime}(\alst)$ is not an integer for any $m\in\mathbb Z$, because $|\kappa\phi^{\prime}(\al_\star)|=|m \kappa x_0|$ is irrational (see Assumption~\ref{ass:x0}(1)). Consider the set of integers
\be\label{Lm def}
L(m):=\{l\in \mathbb Z:\, l=\kappa\phi^{\prime}(\al)\text{ for some }|\al|\le\pi\},\ m\not=0.
\ee
Clearly, $l\in L(m)$ is equivalent to $|l|\le l_\star(m)$. Define the set
\be\label{exc intervals}\begin{split}
&I_{\star}=\{|\al|\le\pi: |\kappa\phi^{\prime}(\al)|\ge l_\star(m)+(1/2)\}.
\end{split}
\ee
An illustration of one part of $I_\star$, namely the interval $\{|\al|\le\pi: \kappa\phi^{\prime}(\al)\ge l_\star(m)+(1/2)\}$, is shown in red in Figure~\ref{fig:sine}. The other interval $\{|\al|\le\pi: \kappa\phi^{\prime}(\al)\le -(l_\star(m)+(1/2))\}$ is not visible. It is possible that $I_\star$ is empty.
Additionally, consider the sets:
\be\label{main intervals}\begin{split}
&I_l:=\{\al\in[-\pi,\pi]: |\kappa\phi^{\prime}(\al)-l|\le 1/2\},\ |l|\le l_\star(m),
\end{split}
\ee
see a blue interval in Figure~\ref{fig:sine}. Again, the second part of $I_l$, which is a subset of $[-\pi,0]$, is not visible. For simplicity of notation, the dependence of the sets $I_\star$ and $I_l$ on $m$ is omitted. Each $I_l$ is the union of at most three intervals. As is easily seen, each subinterval that makes up $I_l$ satisfies the assumptions of Lemma~\ref{lem:int1 est}. By construction,
\be\label{union}
\cup_{|l|\le l_\star(m)}I_l \cup I_{\star}=[-\pi,\pi].
\ee

The intervals that make up $I_l$ are ‘regular’ in the sense that each of them satisfies the assumptions of Lemma~\ref{lem:int1 est}, so the corresponding integrals can be estimated using \eqref{int1 est}. Each of the two intervals that make up $I_{\star}$ is ‘exceptional’: \eqref{int1 est} does not apply to them, because there is no $l$ such that $\kappa  \phi^{\prime}(\al_0)=l$ for some $\al_0\in I_\star$.  
Thus, in addition to the regular sets $I_l$, we have to consider the exceptional set $I_{\star}$. Since Lemma~\ref{lem:int1 est} does not apply to $I_\star$, estimation of its contribution requires special handling. The following two lemmas are proven in Appendix~\ref{sec:two lems}.

\begin{lemma}\label{lem:exc cases} Under the assumptions of Theorem~\ref{lem:Lyapunov} one has
\be\begin{split}
\label{exc II}
\e\sum_{1\le |m|\le\e^{-\ga}}\biggl|\sum_{\al_k\in I_\star}\tilde g_m(\al_k)e\biggl(\frac{\phi(\al_k)}\e\biggr)\biggr|
&=O(\e)\text{ if } M>\nu+1.
\end{split}
\ee
\end{lemma}

\noindent Let the right side of \eqref{int1 est}, where $g$ is replaced by $\tilde g_m$, be denoted $W_{l,m}$. Then we have the following lemma.

\begin{lemma}\label{lem:generic} Under the assumptions of Theorem~\ref{lem:Lyapunov} one has
\be\label{sum-all}
\sum_{1\le |m|\le\e^{-\ga}}\sum_{|l|\in l_\star(m)} W_{l,m}=O(\e^{1/2})\text{ if } M>\max((\nu+7)/4,5/2).
\ee
\end{lemma}

\noindent Combining \eqref{exc II} and \eqref{sum-all} finishes the proof of \eqref{lim De}.

\section{Numerical experiments}\label{sec:numerics}

In this section, we present numerical experiments to verify the main results in section~\ref{sec:setting_mainres}. To do this, we apply \eqref{recon} to simulated noise draws, $\eta_{k,j}$, under the assumption that the useful signal is zero (see section~\ref{sec:setting_mainres}).

In the examples presented here, the entries $\eta_{k,j}:=\eta(\al_k,p_j)$ are drawn from a uniform distribution with mean zero and variance
\begin{equation}
\label{sigma_2_ex}
\sigma^2(\alpha,p) = (1/3)u(\al,p),\ u(\al,p):=[1+(1/2)\sin(\alpha)][1+(1/2)\sin(\pi p)].
\end{equation}
Specifically, we drew random numbers uniformly on $[-1,1]$ using the Matlab function ``rand," and scaled these by $\sqrt{u(\al,p)}$ to generate $\eta_{k,j}$ with sample variance $\sigma^2$ as in \eqref{sigma_2_ex}. Throughout the simulations presented, we set $\epsilon = 1/j_m$, where $j_m\geq 1$ is the sampling rate, and $\kappa$ (as in \eqref{params}) is set to $\kappa = 2\pi$. The reconstruction space is $[-1,1]^2$, $p_j = -1+ (j-1)/j_m$, $1\leq j\leq 2j_m+1$, and $\alpha_k = 2\pi \frac{k}{j_m}$, for $1\leq k\leq j_m$. 
\begin{figure}[!h]
\centering
\begin{subfigure}{0.4\textwidth}
\includegraphics[width=0.9\linewidth, height=6cm, keepaspectratio]{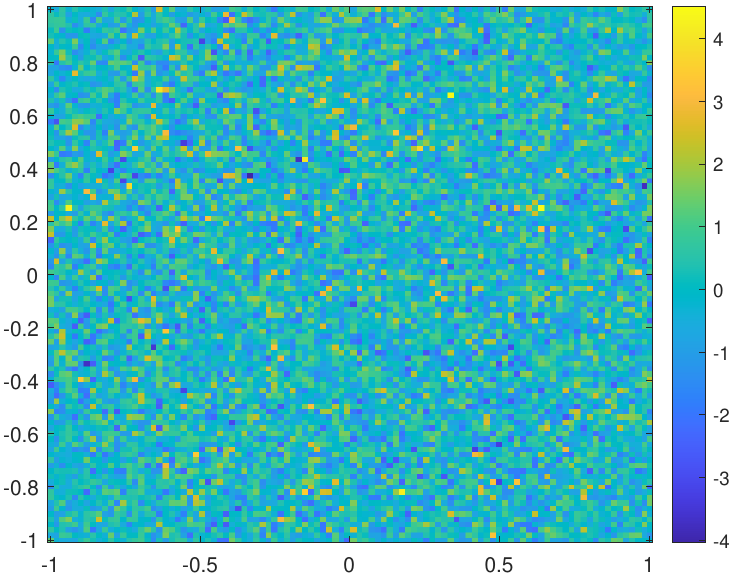}
\subcaption{reconstruction on $[-1,1]^2$} \label{F1a}
\end{subfigure}
\begin{subfigure}{0.4\textwidth}
\includegraphics[width=0.9\linewidth, height=6cm, keepaspectratio]{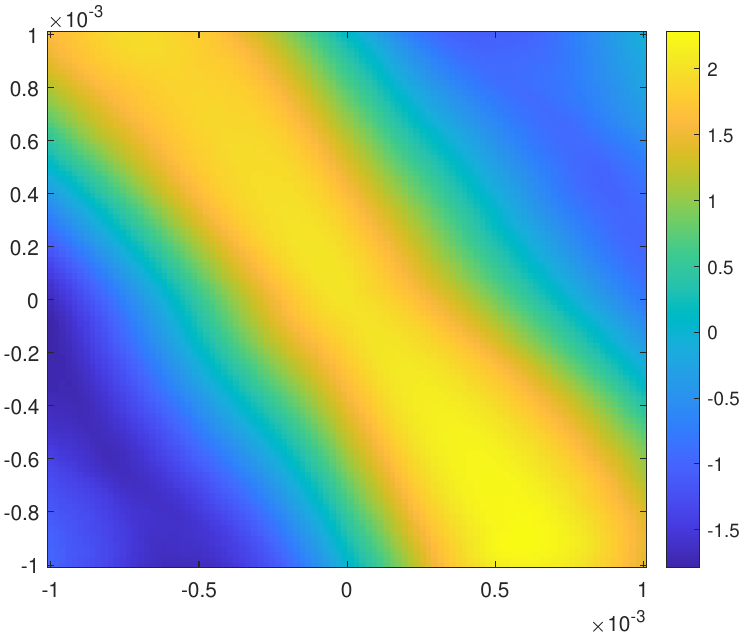} 
\subcaption{reconstruction on $[-\epsilon,\epsilon]^2$} \label{F1b}
\end{subfigure}
\caption{Example image reconstructions from uniform random noise with $j_m = 10^3$ on the full scale (a), and within an $\epsilon$ neighborhood of zero (b). }
\label{F1}
\end{figure}

For reconstruction, we use \eqref{recon}, where $\ik$ is the Keys kernel \cite{Keys1981}.
See figures \ref{F1a} and \ref{F1b}, where we have shown example image reconstructions from a uniform noise draw (as detailed above) on the full image scale (i.e., $[-1,1]^2$), and within an $\epsilon$ neighborhood of zero, respectively. 
The image in figure \ref{F1a} is noisy and the pixel values appear to vary independently. In contrast, in figure \ref{F1b}, the image appears smooth. This is in line with Theorem~\ref{GRF_thm}.
\begin{figure}[!h]
\centering
\begin{subfigure}{0.34\textwidth}
\includegraphics[width=0.9\linewidth, height=4cm, keepaspectratio]{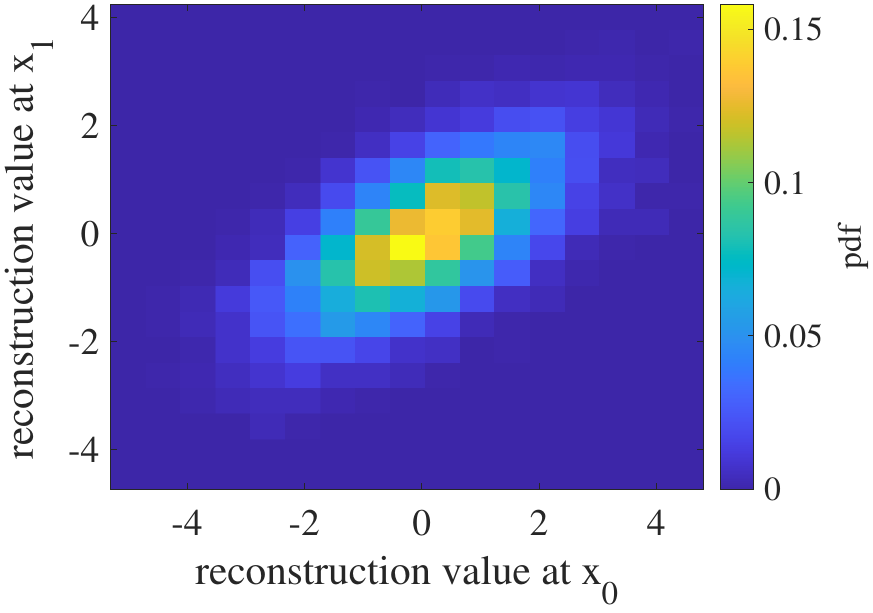}
\subcaption{observed pdf (histogram)} \label{F2a}
\end{subfigure}
\begin{subfigure}{0.34\textwidth}
\includegraphics[width=0.9\linewidth, height=4cm, keepaspectratio]{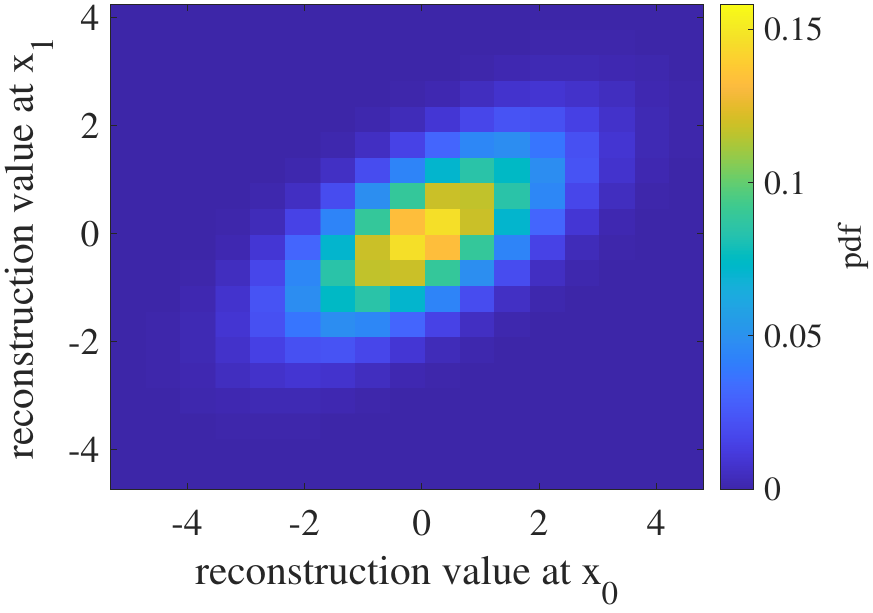} 
\subcaption{predicted pdf} \label{F2b}
\end{subfigure}
\begin{subfigure}{0.30\textwidth}
\includegraphics[width=0.9\linewidth, height=4cm, keepaspectratio]{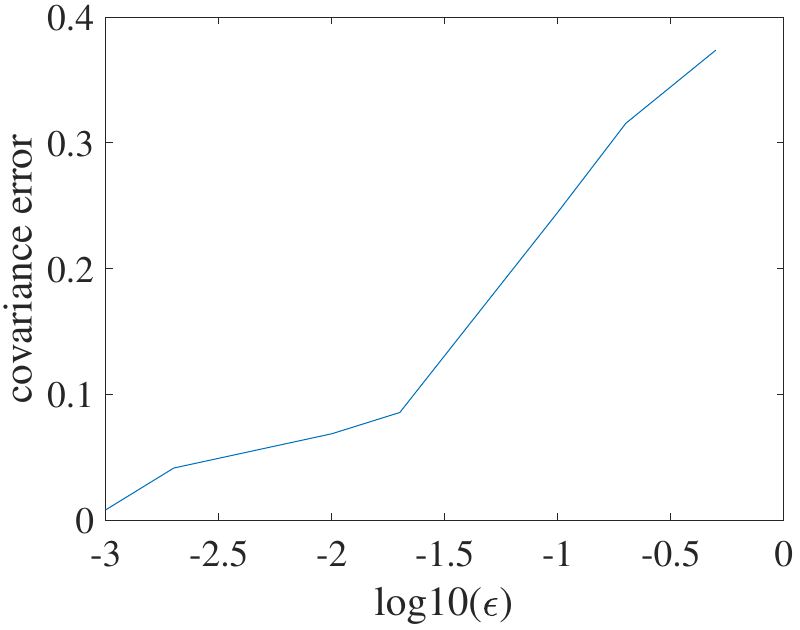}
\subcaption{covariance error with $\epsilon$} \label{F2c}
\end{subfigure}
\caption{Observed (a) and predicted (b) pdf functions for $x_0 = \frac{1}{4}(\sqrt{2},\sqrt{3})$ and $x_1 = x_0 + \frac\e{2\sqrt{2}}(1,1)$, with $j_m = 10^3$ and $n = 10^4$ samples. (c) - covariance error as a function of $\epsilon$; $n = 10^4$ samples were used to calculate the covariance error for all values of $\e$ in the plot in (c).}
\label{F2}
\end{figure}

We now aim to evaluate the accuracy of the covariance predictions provided by \eqref{Cov main}, and show that the reconstructed values in any $O{(\e)}$ size neighborhood follow a Gaussian distribution. To do this, we simulate $n = 10^4$ image reconstructions as in figure \ref{F1b} and calculate the covariance matrix and histogram between two fixed points, $x_0$ and $x_1$ within an $\epsilon$ neighborhood, and match this to the predictions given by $\text{Cov}(\chx,\chy)$ in \eqref{Cov main}. In figure \ref{F2a}, we show the observed pdf function (calculated using a histogram) which corresponds to $x_0 = \frac{1}{4}(\sqrt2,\sqrt3)$ and $x_1 = x_0 + (\epsilon/2)\chx$, where $\chx=(1/\sqrt 2)(1,1)$ and $\epsilon = 10^{-3}$. This matches well with the predicted Gaussian pdf in figure \ref{F2b}, and the least squares error is $\|P_o - P\|_2/\|P_o\|_2 = 0.07$, where $P_o$ and $P$ denote the vectorized images in figures \ref{F2a} and \ref{F2b}, respectively. The observed covariance matrix, $C_o$, and the predicted covariance matrix, $C$, are computed to be
\be
C_o = \begin{pmatrix} 1.34 & 0.86\\ 0.86 & 1.36\end{pmatrix},\quad C = \begin{pmatrix}  \text{Cov}(0,0) & \text{Cov}(0,\chx/2)\\ \text{Cov}(\chx/2,0)  & \text{Cov}(0,0)\end{pmatrix} = \begin{pmatrix} 1.36 & 0.86\\ 0.86 & 1.36\end{pmatrix}, 
\ee
where $\text{Cov}$ and $C$ are calculated using \eqref{Cov main}. We see that $C_o$ and $C$ are very close, and the error is $\|C-C_o\|_F/\|C_o\|_F = 0.01$. The observed mean is $(-0.0148,0.0003)$, which is close to zero as expected since the $\eta_{k,j}$ were drawn from a uniform distribution with mean zero. In this example, $x_1$ is fairly close to $x_0$ (i.e., within distance $\e/2$), as in figure \ref{F1b}, and they have highly correlated reconstructed values.

We note that these results are only valid as $\epsilon \to 0$. To illustrate this, see figure \ref{F2c}, where we have plotted the covariance error (computed using Frobenius norm and the same $x_0$ and $x_1$ points as before) against $\epsilon$. The error is increasing with $\epsilon$, and becomes $>20\%$ when $\epsilon>0.1$ is of a significant size relative to the size of the scanning region (i.e., $[-1,1]^2$). In this case, $\epsilon = 0.1$ is $5\%$ of the scanning region width. 
\begin{figure}[!h]
\centering
\begin{subfigure}{0.4\textwidth}
\includegraphics[width=0.9\linewidth, height=6cm, keepaspectratio]{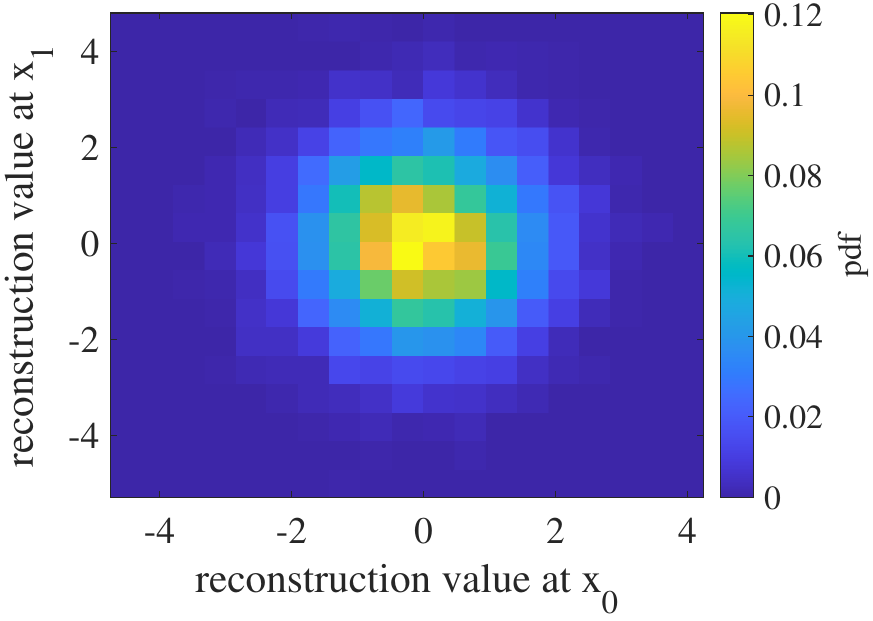}
\subcaption{observed pdf (histogram)} \label{F3a}
\end{subfigure}
\begin{subfigure}{0.4\textwidth}
\includegraphics[width=0.9\linewidth, height=6cm, keepaspectratio]{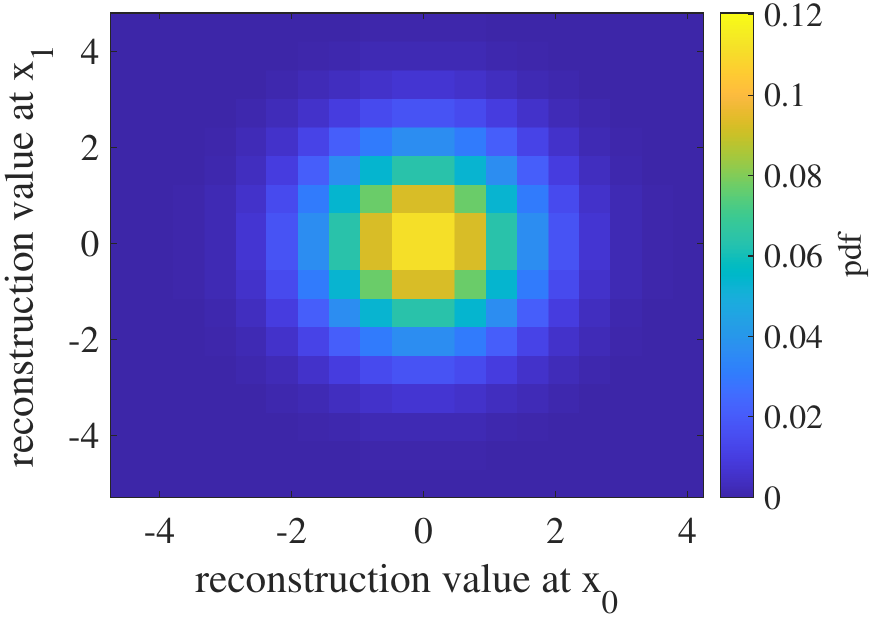} 
\subcaption{predicted pdf} \label{F3b}
\end{subfigure}
\caption{Observed (a) and predicted (b) pdf functions for $x_0 = \frac{1}{4}(\sqrt{2},\sqrt{3})$ and $x_1 = x_0 + \frac{5\epsilon}{\sqrt{2}}(1,1)$, with $j_m = 10^3$ and $n = 10^4$ samples.}
\label{F3}
\end{figure}

In the results thus far, the chosen $x_0$ and $x_1$ were within the same $\epsilon$ neighborhood. For $\epsilon$ small enough, (by Theorem~\ref{GRF_thm}), \eqref{Cov main} holds for pairs of points in larger regions relative to the size of $\e$. To validate this, we keep $x_0$ and $\chx$ the same as before, set $x_1 = x_0+5\e\chx$ with $\epsilon = 10^{-3}$, and calculate the observed and predicted pdf functions as in figure \ref{F2}. We use a sample size of $n = 10^4$, as in the previous example. See figure \ref{F3} for our results. The pdf functions presented in figures \ref{F3a} and \ref{F3b} match up well, and the least squares error is $\|P_o - P\|_2/\|P_o\|_2 = 0.07$, where $P_o$ and $P$ denote the vectorized images in figures \ref{F3a} and \ref{F3b}, respectively, as before. The observed covariance matrix, $C_o$, which corresponds to figure \ref{F3a}, and predicted covariance matrix, $C$, which corresponds to figure \ref{F3b}, are as follows.
\be
C_o = \begin{pmatrix} 1.31 & -0.003\\ -0.003 & 1.32\end{pmatrix},\quad
C = \begin{pmatrix}  \text{Cov}(0,0) & \text{Cov}(0,5\chx)\\ \text{Cov}(5\chx,0)  & \text{Cov}(0,0)\end{pmatrix} = \begin{pmatrix} 1.36 & -0.002\\ -0.002 & 1.36\end{pmatrix}.
\ee
The error again is quite small: $\|C-C_o\|_F/\|C_o\|_F = 0.04$. The observed mean is $(-0.01,-0.01)$, which is near zero as predicted. In this case, $x_0$ and $x_1$ are distance $5\e$ apart and have relatively uncorrelated reconstructed values when compared to the previous example where $|x_0-x_1|=\epsilon/2$. 

The above examples validate our theory and provide accurate predictions for the noise distribution in 2D CT reconstruction in small neighborhoods, e.g., in this case up to size $5\epsilon$.

\appendix

\section{Proof of approximation in equations \eqref{Deps} and \eqref{recon dttpr 2}.}\label{approx_step}

Note that \eqref{Deps} is a particular case of \eqref{recon dttpr 2} with $\chx=\chx_{i_1}=\chx_{i_2}$,  $i_1,i_2\in 1,\dots, K$, and $\vec\theta=(1,0,\dots,0)$, so we prove only \eqref{recon dttpr 2}. Clearly,
\begin{align}\label{c.5}
&\Delta \al\sum_{j,k} \Hc\ik^{\prime}(a_k-j+\vec{\alpha}_k\cdot \chx_{i_1})
 \Hc\ik^{\prime}(a_k-j+\vec{\alpha}_k\cdot \chx_{i_2})\sigma^2(\alpha_k, p_j)
={S}_1+{S}_2,
\end{align}
where $a_k$ are given in \eqref{recon dttpr 3} and
\begin{align*}
    {S}_1&= \Delta \al \sum_{j,k} \Hc\ik^{\prime}(a_k-j+\vec{\alpha}_k\cdot \chx_{i_1})
 \Hc\ik^{\prime}(a_k-j+\vec{\alpha}_k\cdot \chx_{i_2})\big(\sigma^2(\alpha_k, p_j)-\sigma^2(\alpha_k, \vec{\al}_k\cdot x_0)\big), \\
    {S}_2&= \Delta \al \sum_{j,k} \Hc\ik^{\prime}(a_k-j+\vec{\alpha}_k\cdot \chx_{i_1})
 \Hc\ik^{\prime}(a_k-j+\vec{\alpha}_k\cdot \chx_{i_2})\sigma^2(\alpha_k, \vec{\al}_k\cdot x_0).
\end{align*}
Let us define:
\beq
{\Phi}(a,t_1,t_2)=\sum_{j}\lvert \Hc\ik^{\prime}(a-j+t_1)
 \Hc\ik^{\prime}(a-j+t_2) (a-j)\rvert.
\eeq
Using that $\sigma\in C^1([-\pi,\pi]\times[-P,P])$ (Assumption~\ref{noi}(2)) and the property $\Hc\ik^{\prime}(t)=O(t^{-2})$, $t\to\infty$, we bound $S_1$ by:
\begin{align} \label{c.6}
\lvert {S}_1\rvert \leq c\epsilon \Delta\al\sum_{\lvert\al_k\rvert\leq \pi} {\Phi}(a_k,\vec{\alpha}_k\cdot \chx_{i_1},\vec{\alpha}_k\cdot \chx_{i_2}) =O(\e).
\end{align}
Using \eqref{c.5} and \eqref{c.6} in the expression for $\Eb\xi^2_{\e}$, we get \eqref{recon dttpr 2}.

\section{Proof of Lemma~\ref{lem:int1 est} and auxiliary results}\label{sec:first two lems}

\subsection{Proof of Lemma~\ref{lem:int1 est}}\label{ssec:prf int1 est}
We may assume without loss of generality that $\fks(\al_0)=0$ in \eqref{int1 est}. Otherwise, this can be  achieved by subtracting $\fks(\al_0)$ from $\fks(\al)$ in the exponent and not changing the value on the left. 
Then $\fks(\al)$ does not change sign on $I$. We may assume without loss of generality that $\fks(\al)\ge 0$ on $I$. Define the function $E(\al):=\int_0^\al e(t^2)\dd t$. Using the definition of $h_l$ (see \eqref{h-fn}):
\be\label{int1_1 est}\begin{split}
&J:=\int_{I} g(\al) h_l(\al) e\biggl(\frac{\fks(\al)}\e\biggr)\dd \al
=2\pi\kappa \e^{1/2}\int_{I} \frac{g(\al)\fks^{1/2}(\al)}{\sin(\pi\kappa \phi^{\prime}(\al))}\dd E((\fks(\al)/\e)^{1/2}).
\end{split}
\ee
Since $|E(t)|$ is bounded on $\br$, integrating by parts gives
\be\label{int1_2 est}\begin{split}
&|J|\le c \e^{1/2}
\biggl[\max_{\al\in I}\bigg|\frac{g(\al)\fks^{1/2}(\al)}{\sin(\pi\kappa \phi^{\prime}(\al))}\bigg|
+\int_{I}\bigg|\bigg[\frac{g(\al)\fks^{1/2}(\al)}{\sin(\pi\kappa \phi^{\prime}(\al))}\bigg]^{\prime}\bigg|\dd\al.
\end{split}
\ee
The following lemma is proven in appendix~\ref{ssec:sign change}.
\begin{lemma}\label{lem:sign change} The function $f(\al):=[\fks^{1/2}(\al)/\sin(\pi\kappa \phi^{\prime}(\al))]^{\prime}$ may change sign at most finitely many times on $I$ uniformly in $m\not=0$. 
\end{lemma}

Using the last lemma in \eqref{int1_2 est} gives
\be\label{int1_3 est}\begin{split}
&|J|\le c \e^{1/2}
\biggl[g_{\text{mx}}(I)+\int_I|g^{\prime}(\al)|\dd\al\biggr]\max_{\al\in I}\bigg|\frac{\fks^{1/2}(\al)}{\sin(\pi\kappa \phi^{\prime}(\al))}\bigg|.
\end{split}
\ee
Since $|\kappa \phi^{\prime}(\al)|\le 1/2$, $\al\in I$,
\be\begin{split}\label{ratio bnd}
\max_{\al\in I}\bigg|\frac{\fks^{1/2}(\al)}{\sin(\pi\kappa \phi^{\prime}(\al))}\bigg|
\le c\max_{\al\in I}\bigg|\frac{\fks^{1/2}(\al)}{\phi^{\prime}(\al)}\bigg|
\le c\big(\phi^{\prime\prime}(\al_l^*)\big)^{-1/2}.
\end{split}
\ee
Here we used the following lemma with $f(\al)=\fks(\al)$. 
\begin{lemma}\label{lem:aux ineq} If $I$ be an interval such that $\sin\al\not=0$ for any $\al$ in its interior. Define
\be\label{faux def}
f(\al):=\sin\al-\sin\al_0-\cos\al_0(\al-\al_0)
\ee
for some $\al_0\in I$ such that $\sin\al_0\not=0$. One has
\be\label{fineq}
\max_{\al\in I} [|f(\al)|/(f^{\prime}(\al))^2] \le c|f^{\prime\prime}(\al_0)|^{-1}.
\ee
\end{lemma}
Lemma~\ref{lem:aux ineq} is proven in appendix~\ref{ssec:aux ineq}. Note that \eqref{fineq} applies to $\phi_l$ because the inequality in \eqref{fineq} does not change if $f$ is replaced by $cf$ and we can shift the argument $\al-\alst\to\al$. See also the first two sentences in this subsection.

Combining \eqref{int1_1 est}--\eqref{ratio bnd} finishes the proof of lemma~\ref{lem:int1 est}.

\subsection{Proof of Lemma~\ref{lem:sign change}}\label{ssec:sign change}
To prove the lemma we show that there are finitely many solutions to the equation $f(\al)=0$, $\al\in I$, uniformly in $m\not=0$. By simple calculations, we get that these solutions are obtained by solving
\be\label{alt eq}
\frac{\tan(\pi\kappa\phi^{\prime}_l(\al))}{\pi\kappa\phi^{\prime}_l(\al)}=2\frac{\phi_l(\al)\phi_l^{\prime\prime}(\al)}{(\phi^{\prime}_l(\al))^2},\ \al\in I.
\ee 
Set 
\be
v(\al):=(\vec\al-\vec\al_0)\cdot x_0-(\vec\al_0^\perp\cdot x_0)(\al-\al_0). 
\ee
Recall that by our convention, $\pa_\al\vec\al=\vec\al^\perp$. Clearly, $v(\al)$ is an entire function of $\al\in\mathbb C$. By \eqref{h-fn}, $\phi_l(\al)-\phi_l(\al_0)\equiv -m v(\al)$, $\al\in I$. Then \eqref{alt eq} becomes
\be\label{alt eq v2}
\frac{\tan(m\pi v^{\prime}(\al))}{m\pi v^{\prime}(\al)}=2\frac{v(\al)v^{\prime\prime}(\al)}{(v^{\prime}(\al))^2},\ \al\in I.
\ee 
By assumption, $v^{\prime\prime}(\al)\not=0$ in the interior of $I$. Hence we can express both sides of \eqref{alt eq v2} as functions of $s=mv^{\prime}(\al)$ to obtain
\be\label{alt eq v3}
\frac{\tan(\pi s)}{\pi s}-1=v_1(s/m)-1,\ |s|\le1/2.
\ee 
The inequality $|s|\le 1/2$ follows from assumption 3 of Lemma~\ref{lem:int1 est}. Both sides of \eqref{alt eq v3} are analytic near $s=0$ and equal zero at $s=0$ (i.e., $\al=\al_0$). As $m\to\infty$, the right-hand side converges uniformly to zero. The left-hand side equals zero only near $s=0$. Hence, for large $|m|$, all solutions to \eqref{alt eq v3} are located near $s=0$. Expanding both sides near $s=0$ gives
\be\label{alt eq exp}
c_1 s^2(1+O(s^2))=c_2(s/m)^k(1+O(|s/m|)),\ s\to0.
\ee
Here $c_1,c_2\not=0$, and $k\ge 1$ is the index of the first non-zero term in the expansion of $h_1-1$. If $k\ge 2$, the equation has no solutions if $|m|\gg1$. In the remaining case we get the equation
\be\label{two cases}
s=c/m+O(|s/m|+|s|^3)\text{ if }k=1,
\ee
for some $c\not=0$. Therefore, if $|m|$ is sufficiently large, there is one solutions if $k=1$ and no solutions - if $k\ge2$. If $|m|$ is bounded, the number of solutions is bounded as well since \eqref{alt eq} is an equality of two analytic functions.

\subsection{Proof of Lemma~\ref{lem:aux ineq}}\label{ssec:aux ineq}

Without loss of generality we can assume that $I\subset[0,\pi]$, $f^{\prime\prime}(\al)<0$ and, therefore, $f(\al)\le 0$ on $I$. Then \eqref{fineq} becomes
\be\label{faux st11}
\frac{\sin\al_0+\cos\al_0(\al-\al_0)-\sin\al}{(\cos\al_0-\cos\al)^2}\le c\frac1{\sin\al_0},\quad 0<\al,\al_0<\pi.
\ee
The denominator on the left can be zero only when $\al=\al_0$. Hence we consider the case $|h|\ll1$, where $h=\al-\al_0$. Writing $\al=\al_0+h$ and ignoring irrelevant constants gives
\be\label{faux st12}
\frac{\sin\al_0(1-\cos h)+\cos\al_0(h-\sin h)}{h^2\sin^2((\al+\al_0)/2)}\le c\frac1{\sin\al_0},\quad 0<\al,\al_0<\pi,\ |\al-\al_0|\ll 1.
\ee
Thus, suffices it to prove the following two inequalities
\be\label{faux st2}
\frac{\sin^2\al_0}{\sin^2((\al+\al_0)/2)}\le c,\
\frac{|\al-\al_0|\sin\al_0}{\sin^2((\al+\al_0)/2)}\le c,
\quad 
0<\al,\al_0<\pi,\ |\al-\al_0|\ll 1.
\ee
The denominator may approach zero only if $\al,\al_0\to0$ or $\al,\al_0\to\pi$. The two cases are analogous, so we only consider the former. The two inequalities now easily follow from the following inequalities:
\be\label{faux st3}
\frac{\al_0^2}{(\al+\al_0)^2}\le 1,\
\frac{|\al-\al_0|\al_0}{(\al+\al_0)^2}\le 1,
\quad 
\al,\al_0>0.
\ee

\section{Proofs of Lemma~\ref{lem:exc cases} and \ref{lem:generic}}\label{sec:two lems}

\subsection{Proof of Lemma~\ref{lem:exc cases}}\label{sec:exceptional}

Suppose that $\alst$ satisfies $\phi^{\prime\prime}(\alst)=0$, $\alst$ is a local maximum of $\phi^{\prime}(\al)$ and
$l_\star=\lfloor\kappa \phi^{\prime}(\alst)\rfloor$. The dependence of $l_\star$ on $m$ is omitted for simplicity. 
The sets $\{|\al|\le\pi: \kappa\phi^{\prime}(\al)\ge l_\star+(1/2)\}$ and $\{|\al|\le\pi: \kappa\phi^{\prime}(\al)\le -(l_\star+(1/2))\}$ are completely analogous. Therefore, in this section we consider only the former and denote it $I_\star$.

Even though $\kappa\phi^{\prime}(\al)$ may take the value $l_\star$ at two points (on either side of $\alst$), in this section we assume that $\al_{l_\star}^*$ is the smaller of the two (i.e., $\al_{l_\star}^*<\alst$). If $I_\star\not=\varnothing$, there exists $\al_h<\alst$ such that $\kappa\phi^{\prime}(\al_h)=l_\star+(1/2)$. Split $I_\star$ into two intervals $[\al_h,\al_\star]$ and $[\alst,2\alst-\al_h]$. The two intervals are completely analogous, so we prove \eqref{exc II} by restricting the interior sum to $\al_k\in [\al_h,\al_\star]$. 

The sum with respect to $\al_k\in [\al_h,\alst]$ reduces to an integral over $[\al_h,\alst]$ by using $l=l_\star+1$ in \eqref{sum est}. Note that Lemma~\ref{lem:sum1 est} applies regardless of whether there exists an $\al\in I$ such that $\kappa\phi^{\prime}(\al)=l$. The following lemma proves Lemma~\ref{lem:exc cases}.

\begin{lemma}\label{lem:exc cases alt2} Under the assumptions of Theorem~\ref{lem:Lyapunov}, one has
\be\begin{split}
\label{exc II alt2}
\sum_{1\le |m|\le\e^{-\ga}}&\biggl|\int_{\al_h}^{\alst} (\tilde g_m h_{l_\star+1})(\al) e\biggl(\frac{\phi_{l_\star+1}(\al)}\e\biggr)\dd \al\biggr|=O(\e)\text{ if }
M>\nu+1.
\end{split}
\ee
\end{lemma}

\begin{proof} 
Let $J$ be the integral in \eqref{exc II alt2}. Integration by parts in \eqref{exc II alt2} gives:
\be\label{Jn-eps def v2}\begin{split}
|J|&\le c\e (J_1+J_2),\\
J_1&:=\frac{|\tilde g_m(\alst)|}{|\sin(\pi\kappa \phi_{l_\star+1}^{\prime}(\alst))|}+\frac{|\tilde g_m(\al_h)|}{|\sin(\pi\kappa \phi_{l_\star+1}^{\prime}(\al_h))|},\
J_2:=\int_{\al_h}^{\alst}\biggl|\pa_\al\frac{\tilde g_m(\al)}{\sin(\pi\kappa \phi_{l_\star+1}^{\prime}(\al))}\biggr|\dd\al.
\end{split}
\ee
By \eqref{phi pr alt},
\be\label{phi cos}\begin{split}
-\kappa\phi_{l_\star+1}^{\prime}(\al)=&-\cx m\cos(\al-\alst)+(l_\star+1)
=\cx m(1-\cos(\al-\alst))-\cx m+(l_\star+1)\\
=&\Delta+\cx m(1-\cos(\al-\alst)),\ \Delta:=\lceil \cx m\rceil -\cx m,\ l_\star=\lfloor \cx m\rfloor,\ \al\in I_\star,
\end{split}
\ee
where $\lceil r\rceil:=1-\lfloor r\rfloor$ is the ceiling function. Since $\kappa |\phi_{l_\star}^{\prime}(\al_h)|=\kappa |\phi_{l_\star+1}^{\prime}(\al_h)|=1/2$, \eqref{four-coef-bnd} implies
\be\label{exc2 integr term v2}\begin{split}
J_1\le c\rho(m)\langle m\kappa\vec\al_\star^\perp\cdot x_0\rangle^{-1}.
\end{split}
\ee
Also, by \eqref{four-coef-bnd} and \eqref{phi cos},
\be\begin{split}
J_2&\le c\rho(m)\bigg(\int_{\al_h}^{\alst}\frac{|m|}{|\sin(\pi\kappa \phi_{l_\star+1}^{\prime}(\al))|}\dd\al
+\biggl|\int_{\al_h}^{\alst}\pa_\al\frac{1}{\sin(\pi\kappa \phi_{l_\star+1}^{\prime}(\al))}\dd\al\biggr|\bigg)\\
&\le c\rho(m)\bigg(\int_{\al_h}^{\alst}\frac{|m|\dd\al}{\Delta+\cx m(1-\cos(\al-\alst))}
+\frac{1}{|\phi_{l_\star}^{\prime}(\alst)|}\bigg)\\
&\le c\rho(m)\big[(|m|/\langle m\kappa\vec\al_\star^\perp\cdot x_0\rangle)^{1/2}+\langle m\kappa\vec\al_\star^\perp\cdot x_0\rangle^{-1}\big]\le c\frac{\rho(m)}{\langle m\kappa\vec\al_\star^\perp\cdot x_0\rangle}.
\end{split}
\ee
In the first line we used that $|\kappa \phi_{l_\star+1}^{\prime}(\al)|\le 1/2$, $\al\in[\al_h,\alst]$, and the functions $\cos(\pi\kappa \phi_{l_\star+1}^{\prime}(\al))$ and $\phi^{\prime\prime}(\al)\equiv\phi_{l_\star+1}^{\prime\prime}(\al)$ do not change sign on that interval. Thus, the bounds for $J_1$ and $J_2$ are the same. Adding these bounds over $1\le |m|\le e^{-\ga}$ we see that the sum is finite if $M>\nu+1$. 
\end{proof}

\subsection{Proof of Lemma~\ref{lem:generic}}\label{sec:generic}

Throughout this subsection, $\al_r^*$ denote locally unique solutions of $\kappa\phi^{\prime}(\al)=r$, $|r|\le |m \kappa x_0|$. 
 
From \eqref{four-coef-bnd} and \eqref{int1 est},
\be\label{easier bnd}
\biggl|\int_{I} g(\al) h_l(\al) e\biggl(\frac{\fks(\al)}\e\biggr)\dd \al\biggr|\le W_{l,m}\le
c\frac{\rho(m)|m|\e^{1/2}}{(\phi^{\prime\prime}(\al_l^*))^{1/2}},\ m\not=0.
\ee
We sum the right-hand side of \eqref{easier bnd} over $|l|\le l_{\star}(m)$ and then over $1\le|m|\le\e^{-\ga}$. 
Begin by looking at the sums $\sum_{|l|\le l_{\star}(m)}(\phi^{\prime\prime}(\al_l^*))^{-1/2}$. 
Recall that (see Figure~\ref{fig:sine}),
\be
I_l=[\al_{l-(1/2)}^*,\al_{l+(1/2)}^*],\ |l|<l_\star;\
I_l=\begin{cases}
[\al_{l-(1/2)}^*,\al_{l+(1/2)}^*],& |l|=l_\star,I_\star\not=\varnothing,\\
[\al_{l_\star-(1/2)}^*,\alst],& l=l_\star,I_\star=\varnothing,\\
[\alst,\al_{l_\star+(1/2)}^*],& l=-l_\star,I_\star=\varnothing.
\end{cases} 
\ee
Even if $I_\star=\varnothing$ and $\alst\in I_{l_\star}$, Lemma~\ref{lem:int1 est} still applies (e.g., condition 1 of the lemma is not violated even though $\phi^{\prime\prime}(\alst)=0$) because $\alst$ is an endpoint of $I_{l_\star}$, it is not in the interior of $I_{l_\star}$. The same applies to $I_{-l_\star}$.

By \eqref{phi pr alt},
\be\begin{split}
|\phi^{\prime\prime}(\al_l^*)|&=\big[(\cx m)^2-l^2\big]^{1/2}
\ge c|m|^{1/2}
\begin{cases}
\langle m\kappa\vec\al_\star^\perp\cdot x_0\rangle^{1/2},& |l|=l_\star,\\
(l_{\star}-|l|)^{1/2},& |l|<l_\star.
\end{cases}
\end{split}
\ee
Then,
\be\label{sum 21}\begin{split}
\sum_{|l|\le l_\star(m)}\frac1{(\phi^{\prime\prime}(\al_l^*))^{1/2}}
&\le \frac{c}{|m|^{1/4}}\biggl[\langle m\kappa\vec\al_\star^\perp\cdot x_0\rangle^{-1/4}+\sum_{l=0}^{ l_\star-1}(l_{\star}-l)^{-1/4}\biggr]\\
&\le c|m|^{-1/4}\bigl[\langle m\kappa\vec\al_\star^\perp\cdot x_0\rangle^{-1/4}+|m|^{3/4}\bigr],
\end{split}
\ee
and 
\be\label{sum 22}\begin{split}
\sum_{1\le|m|\le\e^{-\ga}}&\frac{\rho(m)|m|}{|m|^{1/4}}\bigl[\langle m\kappa\vec\al_\star^\perp\cdot x_0\rangle^{-1/4}+|m|^{3/4}\bigr]
\le c,\quad M>\max((\nu+7)/4,5/2).
\end{split}
\ee
Consequently, (cf. \eqref{easier bnd}) 
\be\label{sum 23}
\sum_{1\le|m|\le\e^{-\ga}}\sum_{|l|\le l_\star(m)} W_{l,m}=O(\e^{1/2}),\ M>\max((\nu+7)/4,5/2),
\ee
and Lemma~\ref{lem:generic} is proven.

\section*{Acknowledgments}
The work of AK was supported in part by NSF grant DMS-1906361. JWW wishes to acknowledge funding support from The V Foundation, The Brigham Ovarian Cancer Research Fund, Abcam Inc., and Aspira Women's Health.

\bibliographystyle{plain}
\bibliography{refs, My_Collection}
\end{document}